\documentclass[11pt]{amsart}
\usepackage{amssymb}
\usepackage{graphicx}
\usepackage{subfigure}
\usepackage[cmtip,all]{xy}
\usepackage{amscd}
\usepackage{psfrag}

\newtheorem{theorem}{Theorem}[section]
\newtheorem{lemma}[theorem]{Lemma}
\newtheorem{corollary}[theorem]{Corollary}
\newtheorem{proposition}[theorem]{Proposition}

\theoremstyle{definition}

\newtheorem{remark}[theorem]{Remark}

\newtheorem{step}[theorem]{\hskip-2pt}

\numberwithin{equation}{section}

\newcommand\boxt[1]{\boxed{\text{\rm #1}\index{boxt}}}
\def\boxit#1{\vbox{\hrule\hbox{\vrule\kern3pt
     \vbox{\kern3pt#1\kern3pt}\kern3pt\vrule}\hrule}}

\newcommand\on[1]{{\rm #1}}
\newcommand\csch{\text{\rm{csch}}}
\newcommand\sech{\text{\rm{sech}}}

\newcommand\x{\times}
\newcommand\rx{\rtimes}

\newcommand\ps{{'\kern-2pt S}}
\newcommand\lpr{{'\kern-2pt}}
\newcommand\lppr{{\kern1pt''\kern-2pt}}
\newcommand\lpe{{\kern1pt '\kern-2pt e}}
\newcommand\lppe{{\kern1pt ''\kern-2pt e}}
\newcommand\lpE{{\kern1pt '\kern-2pt E}}
\newcommand\lppE{{\kern1pt ''\kern-2pt E}}
\newcommand\bs{{\backslash}}
\newcommand\ra{\rightarrow}

\newcommand\lra{\longrightarrow}

\newcommand\bbh{{\mathbb H}}

\newcommand\bbr{{\mathbb R}}

\newcommand\scirc{\kern -2pt\circ\kern -2pt} 

\newcommand\proj{\on{proj}}

\newcommand\isom{\on{Isom}}

\newcommand\gl{\on{GL}}

\newcommand\so{\on{SO}}

\newcommand\sltr{\on{SL}(2,\mathbb R)}

\newcommand\inv{^{-1}}

\newcommand\Pf{\on{P}\kern-2pt f} 

  {\begin{list}{}{%
     \settowidth{\labelwidth}{\textsf{#1}}%
     \setlength{\leftmargin}\labelwidth
     \advance\leftmargin+\labelsep}}
  {\end{list}}

\newcommand{\kbmar}[1]
  {{\kern-4pt$^\spadesuit$}\marginpar{\rightline{\boxed{\text{#1}}}}}

\def\bi{\mathbf i}
\def\bn{\mathbf n}
\def\bx{\mathbf x}

\def\boxit#1{\vbox{\hrule\hbox{\vrule\kern3pt
     \vbox{\kern3pt#1\kern3pt}\kern3pt\vrule}\hrule}}

\def\boxit#1{\vbox{\hrule\hbox{\vrule\kern3pt
     \vbox{\kern3pt#1\kern3pt}\kern3pt\vrule}\hrule}}

\def\rmk#1{\relax}

\def\multilimits@{\bgroup
  \Let@
  \restore@math@cr
  \default@tag
 \baselineskip\fontdimen10 \scriptfont\tw@
 \advance\baselineskip\fontdimen12 \scriptfont\tw@
 \lineskip\thr@@\fontdimen8 \scriptfont\thr@@
 \lineskiplimit\lineskip
 \vbox\bgroup\ialign\bgroup\hfil$\m@th\scriptstyle{##}$\hfil\crcr}
\def\endSb{\crcr\egroup\egroup\egroup}

\def\angles#1{\langle{#1}\rangle}

\def\phm{\phantom{-}}

\def\lprS{{\kern2pt '{\kern-2pt}S}}
\def\ddanger{\kern-4pt ${^{^\clubsuit}}$ \marginpar{\centerline{$^{\sqrt{}}$}}}
\def\danger#1{\kern-4pt ${^{^\clubsuit}}$ \marginpar{#1}}
\def\dangers#1{\kern-4pt ${^{^\clubsuit}}$ \marginpar{\boxt{{#1}}}}

\def\bzero{\mathbf 0}

\def\sono{\so_0(n,1)}
\def\soto{\so_0(2,1)}
\def\asoto{\frak{so}(2,1)}
\def\son{\so(n)}
\def\sot{\so(2)}

\def\SP{\mathcal H}
\def\HY{\mathbb  H}
\def\bv{\mathbf v}
\def\bw{\mathbf w}
\def\bu{\mathbf u}
\def\zh{{\hat z}}

\def\shorteq{\kern-4pt=\kern-4pt}

\begin{document}
\title{$\text{SO}(n)\backslash\text{SO}_0(n,1)$ has positive curvatures}
\author{Taechang Byun}
\address{University of Oklahoma, Norman, OK 73019, U.S.A.}
\curraddr{} \email{tcbyun@math.ou.edu}
\author{Kyeonghee Jo}
\address{Mokpo National Maritime University, Korea}
\curraddr{} \email{khjo@mmu.ac.kr}
\author{Kyung Bai Lee}
\address{University of Oklahoma, Norman, OK 73019, U.S.A.}
\thanks{The second named author is paritally supported by
Mokpo National Maritime University, 2010.}
\curraddr{} \email{kblee@math.ou.edu}

%
\date{} \maketitle

\begin{abstract}
The Lie group $\text{SO}_0(n,1)$ has the left-invariant metric coming from
the Killing-Cartan form. The maximal compact subgroup $\text{SO}(n)$ of
the isometry group acts from the left. The geometry of the quotient 
space of the homogeneous submersion $\text{SO}_0(n,1)\rightarrow 
\text{SO}(n)\backslash\text{SO}_0(n,1)$ is investigated. 
The space is expressed as a warped product. Its group of isometries and
sectional curvatures are calculated.
\end{abstract}

\setcounter{section}{-1}
\section{Introduction}

On the Lie group  $G=\sono$, we give a left-invariant metric which
comes from the Killing-Cartan form. The maximal compact subgroup
$\son$ $=\son\x\{1\}$ is denoted by $K$.
Then the group of isometries is
$$
\isom_0(G)=G\x K,
$$
the left translation by $G$ and the right translation by $K$. Thus, there
are two actions of $K$, $\ell(K)\subset G$ and $r(K)=K$.
\bigskip

The homogeneous Riemannian submersion by the isometric $r(K)$-action
(which is free and proper)
$$
\son\ra\sono\ra \sono/\son
$$
is very well understood; $\sono/\son$ is the $n$-dimensional hyperbolic space
$\bbh^n$.

It is the purpose of this paper to study the homogeneous Riemannian
submersion by the $\ell(K)$-action
$$
\son\ra\sono\ra \son\bs\sono.
$$

It can be seen that this space $\SP^n=\son\bs\sono$ is diffeomorphic to
$\bbh^n$, but metrically it is not as nice as the case of right
actions. More specifically, it will be shown that
the metric is not conformal to $\bbh^n$, and the space has fewer symmetries.
The following facts will be proven:
\bigskip

1. $\isom_0(\so(n)\bs\so_0(n,1))=r(\so(n))$, and it has one fixed point
$\{\bi\}$, (Theorem \ref{isom-sono}).

2. $\SP^n-\{\bi\}$ is a warped product $(1,\infty)\x_{e^{2\phi}} S^{n-1}$,
(Theorem \ref{warped-prod}).

3. The sectional curvature $\kappa$ satisfies:
$0<\kappa\leq 5$, and $\kappa=5$ is achieved only at $\bi$,
(Theorem \ref{sect-curvature-n}).
\bigskip
\bigskip

\section{Iwasawa  Decomposition}

\step
We shall establish some notation first.
Let
$$
J=\left[\begin{matrix} -I_p & 0\\ 0&I_q \end{matrix}\right],
$$
where $I_p$ and $I_q$ are the identity matrices of size $p$ and $q$.
The group $O(p,q)$ is the subgroup of $\gl(p+q,\bbr)$ satisfying
$A J A^t = J$. It has 4 connected components (for $p,q>0$) and we denote
the connected component of the identity by $\so_0(p,q)$.
It is a semi-simple Lie group. The Iwasawa decomposition is
best described on its Lie algebra. We specialize to $\so_0(n,1)$.

\step
Let $e_{ij}$ denote the matrix whose $(i,j)$-entry is 1 and 0 elsewhere.
The standard metric on $\sono$ is given by the orthonormal basis for the
Lie algebra
$$
E_{ij}=\epsilon_{ij} e_{ij}+e_{ji},\quad 1\leq i < j\leq n+1,
$$
where $\epsilon_{ij}=-1$ if $j<n+1$ and $\epsilon_{ij}=1$ if $j=n+1$.

An Iwasawa decomposition $K A N$ is defined as follows.  Let
$$
N_i=E_{i,n}+E_{i,n+1},\text{ for $i=1,2,\dots,n-1$,}
$$ be a basis for the nilpotent Lie algebra $\mathfrak n$;
$A_1=E_{n,n+1}$ be a basis for the abelian $\frak a$.
The compact subalgebra $\frak k=\frak{so}(n)$ is sitting in
$\frak{so}(n+1)$ as blocked diagonal matrices $\frak{so}(n)\oplus (0)$.
For an explicit discussion of such a decomposition using positive roots, 
see, for example, \cite{Kn}.
\bigskip

\step
It is well known that $NA(=AN)$ forms a (solvable) subgroup. As a
Riemannian \emph{subspace}, $NA$ is an Einstein space; i.e., has a Ricci
tensor which is proportional to the metric.
However, our concern here is $NA$, not as
a subspace, but rather as a \emph{quotient space} of $G$ because it
provides a smooth cross-section for both $G\lra G/K$ and  $G\lra K\bs
G$.
\bigskip

\step
From now on, in a slight abuse of notation,
`$r(K)$-action' means the
right action of $K= \so(n)$ on either $\so_0 (n,1)$ or $\SP^n$
under appropriate situations. Also `$\ell(K)$-action' means the left
action of $K= \so(n)$ on either $\so_0 (n,1)$ or $\bbh^n$.
Note that $\SP^n$ (respectively, $\bbh^n$) does not have an
$\ell(K)$-action (respectively, $r(K)$-action).
\bigskip

\section{$\SP^2=\so(2)\bs\so_0(2,1)$} 
\label{sec:SP^2}

\step[{Metric on $\so_0(2,1)$}]
We shall study the case when $n=2$ first, because this is the building
block for the general case.
The orthonormal basis for the Lie algebra $\frak{so}(2,1)$ is
$$
E_{13}=
\left[\begin{matrix}
0 &0 &1\\
0 &0 &0\\
1 &0 &0\\
\end{matrix}\right],\quad
E_{23}=
\left[\begin{matrix}
0 &0 &0\\
0 &0 &1\\
0 &1 &0\\
\end{matrix}\right],\quad
E_{12}=
\left[\begin{matrix}
0 &-1 &0\\
1 &0 &0\\
0 &0 &0\\
\end{matrix}\right].
$$
The Lie algebras for the Iwasawa decomposition are
$$
\frak{k}=\angles{E_{12}},\quad
\frak{a}=\angles{A_{1}}\quad\text{and}\ \
\frak{n}=\angles{N_{1}},
$$
where
$$
A_1=E_{23}\quad\text{and}\ \
N_1=E_{13}+E_{12}.
$$
The corresponding Lie subgroups are denoted by $K$, $A$ and $N$, respectively.
\bigskip

\step[Global trivialization of $\SP^2$]
In order to study {$\sot\bs\soto$}, it is advantageous to use the
notation $\soto=NAK$ rather than $KAN$. That is, every element $p$
of $\soto$ is uniquely written as a product
$$
p=n a k,\quad n\in N,\ a\in A,\ k\in K.
$$
The nilpotent subgroup $N$ is normalized by $A$, and $NA$ forms a subgroup.
We give a global coordinate to $NA$ by
\begin{align}
\label{varphi}
\varphi:\ \bbr\x \bbr^+ &\lra NA\\
\notag (x,y)&\ \mapsto\  e^{x N_1} e^{\ln(y) A_{1}}.
\end{align}

Note that this is different from the restriction of the exponential map
$\exp: \asoto\ra\soto$.
Sometimes we shall suppress $\varphi$ and write $(x,y)$ for $\varphi(x,y)$.
\bigskip

\step[Comparison with $\sltr$]
We use the standard isomorphism of Lie algebras $\frak{sl}(2,\bbr)$ and
$\frak{so}(2,1)$, sending the basis
$$
\left[\begin{matrix}
0 &1\\
0 &0
\end{matrix}\right],\quad
\tfrac12
\left[\begin{matrix}
1 &0\\
0 &-1
\end{matrix}\right],\quad
\tfrac12
\left[\begin{matrix}
0 &-1\\
1 &0
\end{matrix}\right]
$$
to the basis
$$
N_1,\quad A_{1},\quad -E_{12}.
$$
With the above identification $\varphi$ in diagram (\ref{varphi}), we see the
following correspondence:
$$
\bbr\x\bbr^+ \ni
(x,y)
\longleftrightarrow
\left[\begin{matrix}
1 &x\\
0 &1
\end{matrix}\right]
\left[\begin{matrix}
\sqrt{y} &0\\
0 &\tfrac{1}{\sqrt{y}}
\end{matrix}\right]
\longleftrightarrow
e^{x N_1} e^{\ln(y) A_{1}}
\in NA.
$$
For the compact subgroup $\so(2)$, this isomorphism yields a 2-to-1
covering transformation
$$
\left[\begin{matrix}
\cos \tfrac z2 &-\sin \tfrac z2\\
\sin \tfrac z2 &\phm\cos \tfrac z2
\end{matrix}\right]
\longleftrightarrow
\left[\begin{matrix}
\phm\cos z &\sin z &0\\
-\sin z &\cos z&0\\
0&0&1
\end{matrix}\right]
=e^{z(-E_{12})}.
$$
Therefore, in order to conform with the ordinary M\"obius
transformations of $\so(2)$ on the upper half-plane model, the group
$\so(2)\subset\so_0(2,1)$ will be parametrized by $e^{z (-E_{12})}$
rather than by $e^{z E_{12}}$.

\step[Riemannian metric on $\SP^2$]
With the Riemannian metric on $\soto$ induced by the orthonormal basis
$\{E_{13}, E_{23}, E_{12}\}$,
the group of isometries is
$$
\isom_0(\soto)=\soto\x\sot.
$$
The subgroup $\sot\subset\soto$ acts on $\soto$ as left translations,
$\ell(K)$, freely and properly, yielding a submersion. The quotient space
$\sot\bs\soto$ acquires a unique Riemannian metric that makes the projection,
$\proj: \soto\lra\sot\bs\soto$, a Riemannian submersion.
It has a natural smooth (non-metric) cross section $NA$ in $KNA=NAK$.

At any $p\in\soto$, the vector
$\ell(p)_*(E_{ij})$ is just matrix multiplication $p E_{ij}$, and
$$
\{p E_{13}, p E_{23}, p E_{12}\}
$$
is an orthonormal basis at $p$.
The isometric $\ell(K)$-action induces a homogeneous foliation on $\soto$.
The leaf passing through $p$ is $K p$, the orbit containing $p$. Therefore,
the vertical vector is $E_{12} \, p$. We can find a new orthonormal basis
$\{\bv_1,\bv_2,\bv_3\}$, where the last vector $\bv_3$ is the normalized
$E_{12} \, p$. More explicitly, we write $E_{12}p$ as a combination
of the above orthonormal basis:
\begin{align*}
\bu_3&=E_{12}p=g_1(p)\ p E_{13} + g_2(p)\ p E_{23} + g_3(p)\ p E_{12},\\
\intertext{and set}
\bu_1&=-g_3(p)\ p E_{13} + g_1(p)\ p E_{12}.
\end{align*}
Then take the cross product $\bu_3\x \bu_1$ as $\bu_2$. Now normalize
$\{\bu_1,\bu_2,\bu_3\}$ to get $\{\bv_1,\bv_2,\bv_3\}$.

Thus, $\{\bv_1,\bv_2\}$ is an orthonormal basis for the
horizontal distribution to the homogeneous foliation generated by the
$\ell(K)$-action. We want the projection, $\proj: \soto\ra\SP^2=\sot\bs\soto$,
to be an isometry on the horizontal spaces. Since we are using the
global coordinate system
\begin{equation}
\label{trivialization-2}
\xymatrix@C=4.5pc@R=1pc{
\bbr\x\bbr^+\ar[r]^{\varphi}_\cong &NA\ar[r]^{\kern-48pt\proj|_{NA}}_{\kern-48pt \cong} &\SP^2=\sot\bs\soto
}
\end{equation}
on $\SP^2$, we take the projection $T_p(\soto)=T_p(NAK)\ra T_p(NA)$
for $p\in NA$.
Expressing the images of $\bv_1, \bv_2$ by this projection in terms of
$\{\frac{\partial}{\partial x},\frac{\partial}{\partial y}\}$, we get
  \begin{align*}
  \bw_1&=
-\frac{\sqrt{\left(x^2+1\right)^2+y^4}}{\sqrt{2} y} \frac{\partial}{\partial x} \Big|_{(x,y)}
-\frac{\sqrt{2} x \left(x^2+1\right)}{\sqrt{\left(x^2+1\right)^2+y^4}}\
\frac{\partial}{\partial y}\Big|_{(x,y)} \\
  \bw_2&=
y \sqrt{\frac{2 x^2 y^2}{\left(x^2+1\right)^2+y^4}+1}\
\frac{\partial}{\partial y} \Big|_{(x,y)}.
\end{align*}

\bigskip

\begin{proposition}
\label{ON-on-r2}
The Riemannian metric on the quotient of the
Riemannian submersion $\soto\ra \SP^2=\sot\bs\soto$ is given by the
orthonormal basis $\{\bw_1,\bw_2\}$.
\end{proposition}
The space $\SP^2$ is always assumed to have this  metric.
\bigskip

\step[Subgroup $NA$ with the left-invariant metric]
We mention that the left-invariant metric restricted on the subgroup $NA$
yields a space isometric to the quotient $\bbh^2=\so_0(2,1)/\so(2)$:
The subgroup $NA$ with the Riemannian metric induced from that of
$\so_0(2,1)$ has an orthonormal basis
$\{\frac{1}{\sqrt{2}} N_1, \ A_1\}$
at the identity,
while the quotient $\so_0(2,1)/\so(2)$ is isometric to the Lie group
$NA$ with a new left-invariant metric coming from the orthonormal basis
$\{N_1, \ A_1\}$.
These two are isometric by $(x,y)\mapsto (\sqrt{2} x,y)$,
and have the same constant sectional curvatures $-1$.
Similar statements are true for general $n$.
\bigskip

\step[Global trivialization of $\bbh^2$]
With the same global coordinate system
$$
\xymatrix@C=4.5pc@R=1pc{
\bbr\x\bbr^+\ar[r]^{\varphi}_\cong
&NA\ar[r]^{\kern-48pt\proj|_{NA}}_{\kern-48pt \cong} &\bbh^2=\soto/\so(2),
}
$$
$\bbh^2$
has the orthonormal basis
$\{y\frac{\partial}{\partial x},y\frac{\partial}{\partial y}\}$
(on the plane $\bbr\x\bbr^+$).
Then the projection  $T_p(\soto)=T_p(NAK)\ra T_p(NA)$ for $p\in NA$
is a Riemannian submersion, that is, an isometry on the horizontal spaces.
\bigskip

\step[Special point $(0,1)$]
Note also that the vector fields $\{\bw_1,\bw_2\}$  are globally defined
and smooth (including the point $(0,1)$).  This fact is significant because
we shall use the fact that our space with the point $(0,1)$ removed is a
warped product to calculate curvatures etc. Since the curvature is a smooth
function of the orthonormal basis, the curvatures at the  point $(0,1)$
will simply be the limit of the curvature, $\lim_{(x,y)\ra(0,1)}
\kappa(x,y)$.

\bigskip
\step[$r(K)$-action on $\SP^2$ vs. $\ell(K)$-action on $\bbh^2$]
Observe that $r(K)$ normalizes (in fact, centralizes) the left action
$\ell(K)$, and hence, it induces an isometric action on the quotient
{$K\bs G$}. We need to study this isometric $r(K)$-action in detail.

First we consider the isometric action $\ell(K)$ on the hyperbolic
space $\HY^2=G/K$. For $p\in NA$ and $k\in K$, suppose $k\cdot p=p_1
k_1$. Then $k\cdot (p K)=p_1 K$. That is
$$
\ell(k)\cdot \bar p =\bar p_1 \text{ in } G/K \quad\text{ (if $k\cdot p=p_1 k_1$ for some $k_1$)}.
$$
Now for our $r(K)$-action on $\SP^2=K\bs G$, let $p\in NA$ and $k\in
K$. Suppose $p\cdot k=k_2 p_2$. Then $(K p)\cdot k=K p_2$.
That is
$$
r(k)\cdot \bar p =\bar p_2 \text{ in } K\bs G \quad\text{ (if $p\cdot k=k_2 p_2$ for some $k_2$)}.
$$
\begin{proposition}
\label{r-moebius}
In $xy$-coordinate for $\SP^2=\sot\bs\soto$ {\rm(}upper half-plane{\rm)},
the\ isometric $r(K)$-action on $\SP^2$ is given by:

\noindent For\ $\zh=e^{z(-E_{12})}= \left[\begin{matrix}
\phm\cos z &\sin z &0\\ -\sin z &\cos z &0\\ 0&0&1
\end{matrix}\right]
\in K$ and $(x,y)\in \SP^2$,
\begin{multline*}
r(\zh) \cdot (x,y) =\frac{1}{2y}
\Big(-(-x^2+y^2-1)\sin z+2x y\cos z,\Big.\\
\Big.(-x^2+y^2-1) \cos z+2 x y \sin z+x^2+y^2+1\Big).
\end{multline*}
In vector notation,
$$
r(\hat z)\cdot\left[\begin{matrix} x \\ y \end{matrix}\right]
=\left[\begin{matrix} \cos z &-\sin z \\ \sin z &\phm\cos z \end{matrix}\right]
\left(\left[\begin{matrix} x \\ y \end{matrix}\right]
-\left[\begin{matrix} 0 \\ \tfrac{1+x^2+y^2}{2y} \end{matrix}\right]\right)
+\left[\begin{matrix} 0 \\ \tfrac{1+x^2+y^2}{2y} \end{matrix}\right].
$$
\end{proposition}
\bigskip

\step
Note that $r(\hat z)$ is a ``Euclidean rotation'' with an appropriate center.
More precisely, each $(x,y)$ is on the Euclidean circle centered at
$\left(0,\tfrac{1+x^2+y^2}{2y}\right)$ with radius
$\sqrt{x^2+\left(y-\tfrac{1+x^2+y^2}{2y}\right)^2}$, and $r(\hat z)$ rotates the
point $(x,y)$ along this circle. This can be seen by calculations.

The $\ell(K)$-action on $\bbh^2$ is the genuine M\"obius
transformation, and is given by
$$
\hskip-80pt
\ell(\zh) \cdot (x,y)=
\frac{1}{L}
\Big(\left(x^2+y^2-1\right) \sin z+2 x \cos z,\Big.\\
\Big.{2 y}\Big)
$$
with
$$
L={-\left(x^2+y^2-1\right) \cos z+2 x \sin z+x^2+y^2+1}.
$$
The relation between $r(K)$-action on $\SP^2$ and $\ell(K)$-action on
$\bbh^2$ will be stated in Proposition \ref{tau2} more clearly.

\bigskip

\step
Both $r(K)$- and $\ell(K)$-actions have a unique fixed point at
$(0,1)$, and all the other orbits are Euclidean circles
{centered on the $y$-axis}. This implies that the geometry is
completely determined by the geometry
at the points on the $y$-axis (more economically,
on the subset $[1,\infty)$ of the $y$-axis). The orthonormal bases at the points
of $y$-axis are important. From Proposition \ref{ON-on-r2}, we have

\begin{corollary}
\label{onbasis-on-y-axis}
At $(0,y)\in\SP^2$ with $y>1$, the orthonormal system is
\begin{align*}
\bw_1&=-\sqrt{\cosh(2\ln y)} \frac{\partial}{\partial x} \Big|_{(0,y)}\\
\bw_2&=y \frac{\partial}{\partial y} \Big|_{(0,y)}.
\end{align*}
\end{corollary}
\bigskip

With the orthonormal basis on the upper half-plane model given in
Proposition \ref{ON-on-r2}, we can calculate the sectional curvature.

\begin{theorem}
\label{2-dim_sec_curv}
On the space $\SP^2=\sot\bs\soto$, the sectional curvature at $(x,y)$ is
$$
\kappa(x,y)=
\frac{4 y^2 \left(x^4+2 x^2 \left(y^2+1\right)+y^4+3 y^2+1\right)}{\left(x^4+2 x^2   \left(y^2+1\right)+y^4+1\right)^2}.
$$
In particular, $0<\kappa\leq 5$
and the maximum $5$ is attained at the point $(0,1)$.
\end{theorem}

\step
Note that, because of the isometric $r(K)$-action (see Proposition 
\ref{r-moebius}), it is enough to know the
curvatures at the points on the $y$-axis,
$$
\kappa(0,y)=\frac{4 y^2 (1 + 3 y^2 + y^4)}{(1 + y^4)^2}.
$$

As we shall see in Proposition \ref{tau2}, the $r(K)$-orbits
will be the geometric concentric circles centered at $(0,1)$.
These are Euclidean circles with different
centers, see Proposition \ref{r-moebius}. Over these
$r(K)$-orbits, $\kappa(x,y)$ is constant, of course. In fact, on the
geometric circle of radius $|\ln y|$, the curvature is $\kappa(0,y)$.

Here are graphs of the sectional curvatures.
Figure \ref{Curv-2dim}
shows that $\kappa=5$ is the maximum at $(0,1)$. The level
curves are the geometric circles centered at $(0,1)$ of $\SP^2$.

\begin{figure}[!ht]
\centering
\mbox{\subfigure{\includegraphics[width=2.5in]{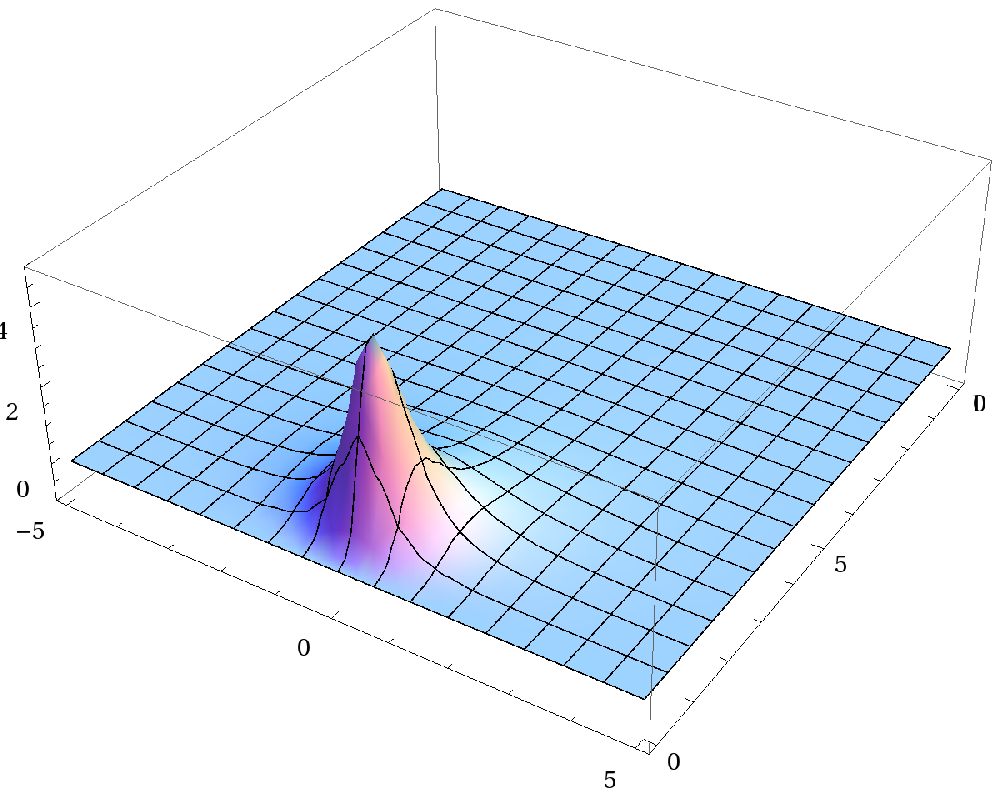}
\label{Curv-2dim}
\quad
\subfigure{\includegraphics[width=2.5in]{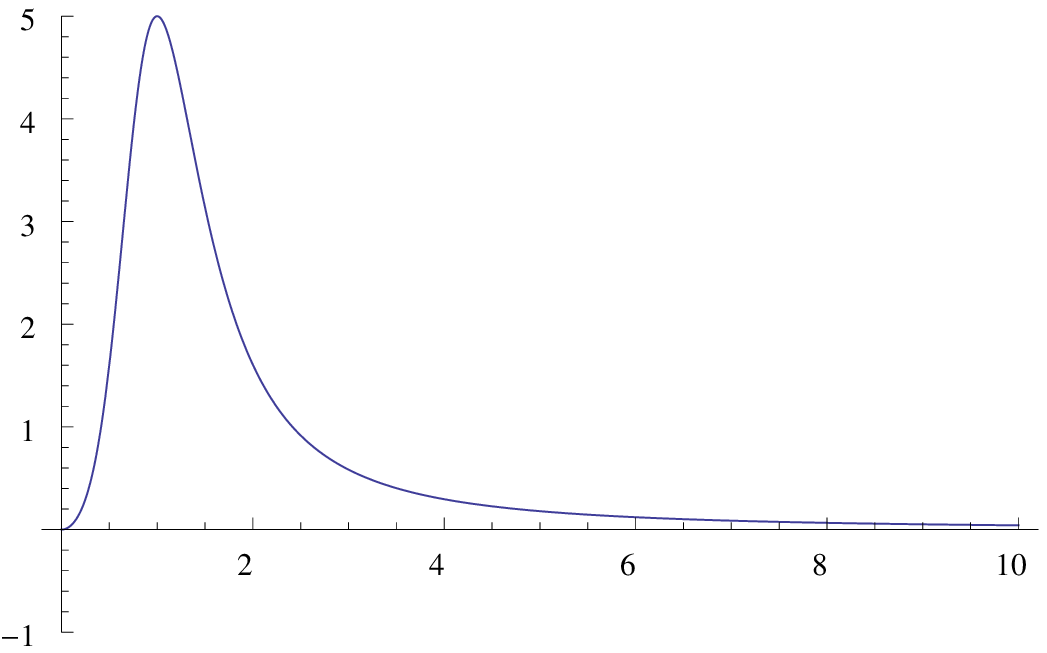} }}}
\caption{$\kappa$ for $(-5<x<5,\ 0<y<10)$, and the cross section at $x=0$}
\end{figure}
\bigskip
\bigskip

\section{$\SP^2=\so(2)\bs\so_0(2,1)$ vs. $\bbh^2=\so_0(2,1)/\so(2)$}

Recall that both spaces $\so(2)\bs\so_0(2,1)$ and $\so_0(2,1)/\so(2)$
have isometric actions by circles, $r(K)$ and $\ell(K)$,
respectively.
The trivialization functions $\varphi$ and $\proj|_{NA}\circ
\varphi$ in diagram (\ref{trivialization-2}) will be suppressed sometimes.

From the weak $G$-equivariant diffeomorphism from $K\bs G$ to $G/K$
given by $K g\mapsto g\inv K$, we can define $\tau:\bbr\x\bbr^+ \lra
\bbr\x\bbr^+$ as in the following

\begin{proposition}
\label{tau2}
For $x\in\bbr^{1}$ and $y\in \bbr^+$,
$(0,y)(x,1)(0,y)\inv =(yx,1)$ 
(with the notation in the diagram {\rm (\ref{trivialization-2})})
so that
$$
(x,y)\inv=(-\tfrac{x}{y},\tfrac{1}{y}).
$$
The map
$$
\tau:\ \SP^2=\so(2)\bs\so_0(2,1)\lra \HY^2=\so_0(2,1)/\so(2)
$$
{\rm(}as a map $\bbr\x\bbr^+ \lra \bbr\x\bbr^+${\rm)}
defined by

$$
\tau(x,y)=(-\tfrac{x}{y},\tfrac{1}{y})
$$
has the following properties:

{\rm (1)}
$\tau$ is a weakly $\so(2)$-equivariant diffeomorphism of period 2.
More precisely,
$$
\tau(r(\zh)\cdot(x,y))=\ell(\zh\inv)\cdot\tau(x,y)
$$
for $\hat{z}\in\so(2)$. In other words, the identification
of $\SP^2$, $\bbh^2$ and  $NA$ with $\bbr \times \bbr^+$ as sets
permits some abuse of $\tau$ and gives the following relation between
$r(K)$-action and $\ell(K)$-action:
$r(\zh)\cdot(x,y)=\tau\left(\ell(\zh\inv)\cdot\tau(x,y)\right)$.

{\rm (2)}
$\tau$ leaves the geometric circles centered at $(0,1)$ in each geometry
invariant. That is, for $m>0$, the Euclidean circle
$$
x^2+(y-\cosh(\ln m))^2=\sinh^2(\ln m)
$$
is a geometric circle centered at $(0,1)$ with radius $|\ln m|$, in both
geometries, and the map $\tau$ maps such a circle to itself.
These circles are $r(K)$-orbits in $\SP^2$ and $\ell(K)$-orbits in $\bbh^2$
(when $\SP^2$ and $\bbh^2$ are identified with $\bbr\x\bbr^+$) at the same time.

{\rm (3)} $\tau$ gives a 1-1 correspondence
between the two sets of all the geodesics passing through $(0,1)$
in the two geometries $\SP^2$ and $\HY^2$. In fact, $\tau$ maps
the $y$-axis to itself and half-circles $\{
(x-\alpha)^2+y^2=\alpha^2+1 \}_{\alpha\in\bbr}$ to hyperbolas $\{
x^2+2\alpha xy-y^2+1=0 \}_{\alpha\in\bbr}$.
\end{proposition}

\begin{proof}
(1) Observe that $\tau(x,y)$ corresponds to the inverse of $\varphi(x,y)$
in the group $N A$. In fact, we have
$$
\varphi(\tau(x,y))=(\varphi(x,y))\inv.
$$
For $\varphi(x,y)\in N A$ and $\hat z\in K$, one can find $k\in K$ for which
$$
k \cdot \varphi(x,y) \cdot \hat{z} \in NA.
$$
Thus,
\begin{align*}
\varphi\big(\tau(r(\hat z)\cdot (x,y))\big)
&=\big(\varphi(r(\hat z)\cdot (x,y))\big) \inv \\
&=(k \cdot \varphi(x,y)\cdot \hat z\big)\inv\\
&=\hat z\inv\cdot (\varphi(x,y))\inv \cdot k \inv\\
&=\hat z\inv\cdot \big(\varphi (\tau (x,y)) \big) \cdot k \inv\\
&=\varphi\big(\ell(\hat z)\inv\cdot \tau(x,y)\big).
\end{align*}

(2)
 Any $(x,y)$ lies on the Euclidean circle centered at $(0,c)$,
where $c=\tfrac{1+x^2+y^2}{2y}$ and radius
$r=\sqrt{x^2+\left(y-\tfrac{1+x^2+y^2}{2y}\right)^2}$. In particular,
$(0,m)$ lies on the Euclidean circle centered at $(0,c)$, where
$c=\cosh(\ln m)$ and radius $r= | \sinh(\ln m) |$. Note $m=\cosh(\ln
m) + \sinh(\ln m)$ and $\tfrac{1}{m}=\cosh(\ln m) - \sinh(\ln m)$,
which show that both $(0,m)$ and $(0, \tfrac{1}{m})$ lie on the same
circle. Then, in $\bbr \times \bbr^+$,
$$
r(K)\cdot(0,m)=\tau(\ell(K)\cdot(0,\tfrac 1m))=\ell(K)\cdot(0,\tfrac
1m)=\ell(K)\cdot(0,m)
$$
shows this circle is both $r(K)$-orbit of the point $(0,m)$ (in
$\SP^2$) and its $\ell(K)$-orbit (in $\bbh^2$) at the same time.
Since both $r(K)$-action on $\SP^2$ and $\ell(K)$-action on $\bbh^2$
are isometric, every point on the circle has the same distance from
$(0,1)$ in each geometry. 

In $xy$-coordinates, the equations
for geodesics in $\SP^2$ are a system of 2 equations

\small{
\begin{align*}
0=&x''(t){\left(2 x (t)^2 y (t)^3+\left(x (t)^2+1\right)^2 y (t)+y (t)^5\right)^2}
\\
&
-2 y (t) x'(t) y'(t)
\Big(
x (t)^6 \left(4 y (t)^2+2\right)+x (t)^4 \left(6 y (t)^4+8 y
(t)^2\right)\Big) \\
&
-2 y (t) x'(t) y'(t)
\Big(
2 x (t)^2 \left(2 y (t)^6+y (t)^4+2 y (t)^2-1\right)+x (t)^8+y
(t)^8-1
\Big) \\
&-4 x (t) y (t)^2 x'(t)^2\left(x (t)^2+1\right)^2 \\ 
&+ x (t) y'(t)^2
\left(4 \left(x (t)^2+1\right) y (t)^6+2 \left(3 x (t)^4+4 x(t)^2+1\right) 
y (t)^4\right)\\
&+ x (t) y'(t)^2
\left(
4 \left(x (t)^2+1\right)^3 y (t)^2+\left(x (t)^2+1\right)^4+y (t)^8\right)\\
\intertext{and}
0=&
{y (t) \left(2 x (t)^2 \left(y (t)^2+1\right)+x (t)^4+y (t)^4+1\right)^2}
y''(t)\\
&
-4 x (t) y (t) x'(t)y'(t)
\left(
x (t)^4 \left(3 y (t)^2+1\right)+x (t)^2 \left(3 y
(t)^4+4 y (t)^2-1\right)\right) \\
&-4 x (t) y (t) x'(t)y'(t)
\left(
x (t)^6+y (t)^6+y (t)^4+y (t)^2-1\right) 
\\
&
+2 y (t)^2  x'(t)^2
\left(
3 x (t)^4 \left(y (t)^2-1\right)+x (t)^2 \left(y
(t)^2+1\right) \left(3 y (t)^2-5\right)\right)\\
&+2 y (t)^2  x'(t)^2
\left(
x (t)^6+y (t)^6-y (t)^4-y
(t)^2-1
\right)
\\
&+
y'(t)^2
\left(
2 x (t)^6 \left(y (t)^2+1\right)+4 x (t)^4 y (t)^2-2 x (t)^2 \left(y
(t)^6+y (t)^4-y (t)^2+1\right)\right) \\
&+
y'(t)^2
\left(
x (t)^8-\left(y (t)^4+1\right)^2
\right). 
\end{align*}
}
\noindent
One can readily check that
$$
\gamma(t)=(0,e^t)\in\SP^2,\ 0\leq t\leq |\ln (m)|=|\ln (\tfrac{1}{m})|
$$
is a unit-speed geodesic, and therefore,
$\text{Length}(\gamma)=|\ln(m)|$.
This is the geometric radius of the circle centered at $\bi=(0,1)$
$\in\SP^2$.

{(3) Let $\mathcal{G}_{\SP^2}$ and
$\mathcal{G}_{\HY^2}$ be the sets of all the unit-speed geodesics
starting from $\bi$ in $\SP^2$ and $\HY^2$, respectively. Then
$$
\mathcal{G}_{\SP^2}= \{r(k)\cdot \gamma(\bullet): \bbr \lra \SP^2 \}_{k\in
  K}
$$ 
and
$$
\mathcal{G}_{\HY^2}= \{l(k)\cdot \gamma(\bullet) : \bbr \lra \HY^2 \}_{k\in K},
$$
since $\gamma\in \mathcal{G}_{\SP^2}\cap \mathcal{G}_{\HY^2}$.}

The 1-1 correspondence between
$\mathcal{G}_{\SP^2}$ and $\mathcal{G}_{\HY^2}$ by $\tau$ comes
from the weak equivariance of $\tau$ and the fact
$\tau(\gamma(t))=\gamma(-t)$. In fact, for $k\in K$ and $t\in
\bbr$,
$$
r(k)\cdot \gamma(t)=\tau(\ell(k^{-1})\cdot \tau(\gamma(t)))
=\tau(\ell(k^{-1})\cdot\gamma(-t))
=\tau(\ell(k^{-1})\cdot \ell(\hat{\pi})\cdot \gamma(t)).
$$
Finally, we can check easily that for each
$\alpha\in\bbr$, the hyperbola $x^2+2\alpha xy-y^2+1=0 $, a  $\SP^2$-geodesic, 
corresponds to the half-circle $(x-\alpha)^2+y^2=\alpha^2+1$, a
$\HY^2$-geodesic. 
\end{proof}
\bigskip

\begin{theorem} \label{prop:warped-2-dim}
The space $\SP^2-\{\bi\}$ is isometric to the warped
product $B \times_{e^{2 \phi}} S^1$, where $B=(1,\infty)=\{(0,y):\
1<y<\infty\} \subset\SP^2$ has the induced metric; that is,
$|\tfrac{\partial}{\partial t}(t_0)| = \tfrac{1}{t_0}$ for $t_0 \in
(1, \infty)$, $S^1$ has the standard metric; and $e^{2 \phi (t)} =
\frac{\sinh ^2 (\ln t)}{\cosh (2\ln t)}$.
\end{theorem}

\begin{proof}
The crucial points are that $r(K)\subset\isom(\SP^2)$ and
that all the other orbits are circles, except for the one fixed point
$\bi=(0,1)$.
This will make our space a warped product of
$S^1$ by the base space $B$, and we need to find a map $\phi$ in $B
\times_{e^{2 \phi}} {S}^1$. The $r(K)$-orbit through $(0,y)\in\SP^2$
is, by Proposition \ref{r-moebius},
$$
r(\zh)\cdot (0,y)=
\big(-\sinh(\ln y) \sin z,\ \sinh(\ln y)\cos z + \cosh(\ln y)\big).
$$
Define a map
$$
f : B \times_{e^{2 \phi}} {S}^1 \lra \SP^2 =\so(2)\bs\so_0(2,1)
$$
by
\begin{align*}
f(t,\zh)
&=f(t,\zh \cdot \hat{0})\\
&=r(\zh ^{-1})\cdot (0,t)\\
&=r(\widehat{-z}) \cdot (0,t)\\
&=\big(\sinh(\ln t) \sin z,\ \sinh(\ln t)\cos z + \cosh(\ln t)\big).
\end{align*}
Note the definition of $f$ does not depend on $e^{2\phi}$ and it is
weakly equivariant with the $r(\so(2))$-action without the concept of
isometry yet. Since $f$ maps the base  $B\x \hat 0$ of the warped
product to the $y$-axis of $\SP^2$, it is enough to find $e^{2\phi}$
which makes $f$ isometric on $B\x \hat 0$.

Recall that $\SP^2$ has an orthonormal basis
$$
\Big\{
  -\sqrt{\cosh(2\ln t)} \frac{\partial}{\partial x} \Big|_{(0,t)}, \
  t \frac{\partial}{\partial y} \Big|_{(0,t)}
\Big\}
$$
at $f(t,\hat0)=(0,t)$,\ $t>1$, see Corollary
\ref{onbasis-on-y-axis}.
Also note that the metric on $B\x_{e^{2\varphi}} S^1$ is given by the
orthonormal basis
$$
\Big\{t\frac{\partial}{\partial t} \Big|_{(t, \zh)},\
-e^{-\phi (t)} \frac{\partial}{\partial \zh} \Big|_{(t, \zh)} \Big\}
$$
at $(t,\zh)$.
Observe
\begin{align*}
f_* \left(\frac{\partial}{\partial t} \Big|_{(t,\hat 0)}\right)
&=\frac{d (f\circ t)}{d t} \Big|_{(t,\hat 0)}\\
&=\frac{\partial}{\partial t}\big(f(t,\zh)\big)\Big|_{z=0}\\
&=\frac{1}{t}\Big(\cosh(\ln t)\sin z
\frac{\partial}{\partial x} \Big|_{f(t,\zh)}\\
&{\hskip48pt}+\big(\cosh(\ln t)\cos z
 +\sinh(\ln t)\big)
\frac{\partial}{\partial y} \Big|_{f(t,\zh)}
\Big)\Big|_{z=0}\\
&=\frac{\partial}{\partial y} \Big|_{f(t,\hat 0)}\\
&=\frac{\partial}{\partial y} \Big|_{(0,t)}\\
\intertext{and, we have}
f_* \left(t\frac{\partial}{\partial t }\Big|_{(t,\hat 0)}\right) &=
t \frac{\partial}{\partial y} \Big|_{(0,t)}.\\
\end{align*}

\noindent
Thus, if
\begin{align}
\label{warp1}
f_* \left(e^{-\phi (t)} \frac{\partial}{\partial \zh} \Big|_{(t,\hat 0)}\right)
&= - \sqrt{\cosh \big(2 \ln t
\big)} \frac{\partial}{\partial x} \Big|_{(0,t)},\\
\intertext{then $f$ will be an isometry. Now,}      
\label{warp2}
f_* \left(e^{-\phi (t)} \frac{\partial}{\partial \zh} \Big|_{(t,\hat 0)}\right)
&=e^{-\phi (t)} \frac{d (f\circ \zh)}{d \zh} \Big|_{(t,\hat 0)}\\
\notag
&=e^{-\phi (t)}\Big(\sinh(\ln t) \cos z \frac{\partial}{\partial x} \Big|_{f(t,\zh)}\\
\notag
&{\hskip48pt}-\sinh(\ln t)\sin z \frac{\partial}{\partial y} \Big|_{f(t,\zh)}\Big)
\Big|_{z=0}\\
\notag
&=e^{-\phi (t)} \sinh(\ln t)\frac{\partial}{\partial x} \Big|_{(0,t)}.
\end{align}
From the equalities (\ref{warp1}) and (\ref{warp2}), the condition is then
$$
-\sqrt{\cosh(2\ln t)}=e^{-\phi (t)} \sinh(\ln t),
$$
which implies $ e^{2 \phi (t)} = \frac{\sinh ^2 (\ln t)}{\cosh(2\ln
t)}$.
\end{proof}
\bigskip

We calculate $\kappa(0,y)$ again using the warped product. The result
conforms with Theorem \ref{2-dim_sec_curv}.

\begin{corollary}
For $(t,0) \in B \times_{e^{2 \phi}} S^1$,
$$
\kappa(t,0)=\frac{4 t^2 (1 + 3 t^2 + t^4)}{(1 + t^4)^2}.
$$
\end{corollary}

\begin{proof}
From $ e^{2 \phi (t)} = \frac{\sinh ^2 (\ln t)}{\cosh (2\ln t)}$, we
get
$$
  \{
    t \tfrac{\partial}{\partial t} \! \mid_{(t,0)} ,
    -\tfrac{\sqrt{\cosh (2 \ln t)}}{\sinh (\ln t)}
      \tfrac{\partial}{\partial \hat{z}} \! \mid_{(t,0)}
  \}
$$
is an orthonormal basis at $(t,0) \in B \times_{e^{2 \phi}}
S^1$ and
$$\phi(t) = \ln (\sinh (\ln t)) - \tfrac{1}{2} \ln (\cosh (2 \ln t)).$$
Since $\phi$ is constant along each circle,
\begin{align*}
 \nabla \phi \! \mid_{(t,0)}
 &= \langle \nabla \phi, t
    \tfrac{\partial}{\partial t}\rangle \,
    t \tfrac{\partial}{\partial t}\mid_{(t,0)} \\
 &= (t \tfrac{\partial \phi}{\partial t}) \,
    t \tfrac{\partial}{\partial t} \mid_{(t,0)} \\
 &= (\coth(\ln t) - \tanh (2 \ln t)) \,
    t \tfrac{\partial}{\partial t} \mid_{(t,0)} .
\end{align*}

For tangent vectors $T_1, T_2 \in T(S^1)$ and $X\in T(S^{1})^\perp$
in the warped product, we have
$$
  R(X, T)Y =
  \big(
  h_{\phi} (X,Y) +
  \langle \nabla \phi , X \rangle \langle \nabla \phi , Y \rangle
  \big)
  T
$$
and so
$$
  \langle R(X, T)T, Y \rangle _{\phi}  =
  - e^{2 \phi} |T|^2 _{S^1}
  \big(
  h_{\phi} (X,X) +  \langle \nabla \phi , X \rangle ^2
  \big),
$$
where $h_{\phi}$ is a hessian form, see \cite[p.60,
Proposition2.2.2, Corollary 2.2.1]{GW}. Since
\begin{align*}
  h_{\phi}(t \tfrac{\partial}{\partial t}, t \tfrac{\partial}{\partial t})
  &= \langle
       \nabla _{t \tfrac{\partial}{\partial t}} \nabla \phi, \,
       t \tfrac{\partial}{\partial t}
     \rangle  \\
  &= - \csch^2 (\ln t) - 2 \, \sech ^2 (2 \ln t) \\
  &= - \csch^2 (\ln t) - 2 + 2 \tanh ^2 (2 \ln t),
\end{align*}

\begin{align*}
 \kappa (
         t \tfrac{\partial}{\partial t} ,
         -\tfrac{\sqrt{\cosh (2 \ln t)}}{\sinh (\ln t)}
            \tfrac{\partial}{\partial \hat{z}}
        )
 &= - \big(
        \langle
          \nabla \phi, t \tfrac{\partial}{\partial t}
        \rangle ^2
        +
        h_{\phi}(t \tfrac{\partial}{\partial t}, t \tfrac{\partial}{\partial t})
      \big) \\
 &= 1 - 3 \tanh ^2 (2 \ln t) - 2 \coth (\ln t) \cdot \tanh (2 \ln t) \\
 &= \tfrac{4 t^2 (1 + 3 t^2 + t^4)}{(1 + t^4)^2}.\qedhere
\end{align*}
\end{proof}

\begin{remark}
The following are well known: the space $\HY^2-\{\bi\}$ is isometric
to the warped product $(0,1) \times_{e^{2 \psi}} S^1$, where
$(0,1)\subset\HY^2$ has the induced metric from $\HY^2$, that
is, $|\tfrac{\partial}{\partial t}(t_0)| = \tfrac{1}{t_0}$ for $t_0
\in (0,1)$; $S^1$ has the standard metric; and $e^{2 \psi (t)}
= {\sinh ^2 (\ln t)}$.\\

  The isometry can be given by
  $$
  \tilde{f} :
(0,1) \times_{e^{2 \psi}} {S}^1\lra \HY^2-\{\bi\}
  $$
  defined by
  $$
  \tilde{f}(s,\hat{u}) = \ell(\hat{u}) \cdot (0, s).
  $$
See, for example, \cite[p.58, Theorem 2.2.1]{GW}.
\end{remark}

\begin{corollary}
The map $\tau$ induces a map on the warped products
$$
\tau': (1,\infty) \x_{e^{2\phi}} S^1 \lra (0,1) \x_{e^{2\psi}} S^1
$$
given by
$$\tau' (t,\hat{z}) = (\tfrac{1}{t}, \hat{z}),$$
which is  $\so(2)$-equivariant and satisfies
$\tilde{f} \circ \tau ' = \tau \circ f$.
\end{corollary}

The following commutative diagram shows more detail:
$$
\CD
(1,\infty) \x_{e^{2\phi}} S^1    @>{\tau'}>>
(0,1) \x_{e^{2\psi}} S^1
\\
@V{f}VV @V{\tilde{f}}VV
\\
\SP^2=\so(2)\bs\so_0(2,1) @>\tau>>
\bbh^2=\so_0(2,1)/\so(2)
\endCD
$$
\vspace{0.5cm}
$$
\CD
(t,\hat{z} \cdot \hat{0})=(t,\hat{z})  @>{\tau'}>>
(\tfrac{1}{t},\hat{z})=(\tfrac{1}{t},\hat{z}\cdot\hat{0})
\\
@V{f}VV @V{\tilde{f}}VV
\\
r(\widehat{-z}) \cdot (0,t) @>\tau>>
\ell(\zh) \cdot (-\tfrac{0}{t}, \tfrac{1}{t})
\endCD
$$

\begin{figure}[!ht]
\centering
\includegraphics[width=2.5in]{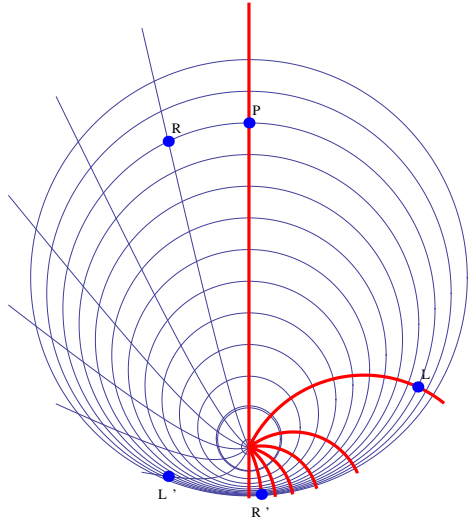}
\caption{Geometric circles and orthogonal geodesics in two geometries.\
$R=   r(\hat{\tfrac{\pi}{7}})\cdot P$,\
$L=\ell(\hat{\tfrac{\pi}{7}})\cdot P$} and $R'=\tau(R)$, $L'=\tau(L)$.
\end{figure}
\bigskip

\section{The general case: $\so(n)\bs\so_0(n,1)$}

\step[Subgroup $NA$ with the left-invariant metric]
As is well known, the subgroup $NA$ has the  structure of a solvable Lie group
$N\rx A$, where
$$
N\cong \bbr^{n-1},\quad A\cong\bbr^+.
$$

The subgroup $NA$ with the Riemannian metric induced from that of
$\so_0(n,1)$ has an orthonormal basis
$$
\{\tfrac{1}{\sqrt{2}} N_1,\ \tfrac{1}{\sqrt{2}}N_2,\ \dots,\
\tfrac{1}{\sqrt{2}}N_{n-1},\ A_{1}\}.
$$
at the identity
while the quotient $\so_0(n,1)/\so(n)$ is isometric to the Lie group
$NA$ with a new left-invariant metric coming from the orthonormal basis
$$
\{N_1,\ N_2,\ \dots,\ N_{n-1},\ A_{1}\}.
$$
These two are isometric by $(\bx,y)\mapsto (\sqrt{2}\bx,y)$,
and have the same constant sectional curvatures $-1$.

\step[Global trivialization of $\bbh^n$]
With the Riemannian metric on $\sono$ induced by the orthonormal basis
$\{E_{ij}:\ 1\leq i < j\leq n+1\}$,
the group of isometries is
$$
\isom_0(\sono)=\sono\x\son.
$$
The subgroup $\son\subset\sono$ acts on $\sono$ as left translations,
$\ell(K)$, freely and properly, yielding a submersion. The quotient space
$\son\bs\sono$ acquires a unique Riemannian metric that makes the projection,
$\proj: \sono\lra\son\bs\sono$, a Riemannian submersion.
It has a natural smooth (non-metric) cross section $NA$ in $KNA=NAK$.

A map
\begin{align*}
\varphi:\ \bbr^{n-1}\x\bbr^+ &\lra NA\\
(\bx,y)&\ \mapsto\  e^{\sum_{i=1}^{n-1}x_i N_i} e^{\ln(y) A_1},
\end{align*}
where $\bx=(x_1,\dots,x_{n-1})$, gives rise to a global
trivialization for the subgroup $NA$ and our space $\SP^n$. Thus, we
shall use $(\bx,y)$ to denote a point in $\SP^n\cong NA$.

\step
Note, for $\bx\in\bbr^{n-1}$ and $y\in \bbr^+$,
$$
(\bzero,y)(\bx,1)(\bzero,y)\inv =(y \bx,1).
$$

Even though we use the local trivialization $\SP^n=\so(n)\bs \sono\ra N A$,
the metric on $\SP^n$ is not related to the group structure of $N A$. That
is, the metric is neither left-invariant nor right-invariant.
\bigskip

\begin{theorem}
\label{isom-sono}
$\isom_0(\so(n)\bs\so_0(n,1))=r(\so(n))$.
\end{theorem}

\begin{proof}
The normalizer of $\ell(\so(n))$ in $\isom_0(\so(n,1))=\ell(\so_0(n,1))\x
r(\so(n))$ is $\ell(\so(n))\x r(\so(n))$.
Since $\ell(\so(n))$ acts ineffectively on the quotient, only $r(\so(n))$
acts effectively on the quotient as isometries.
Thus,
$\isom_0(\so(n)\bs\so_0(n,1))\supset r(\so(n))$.

Suppose these are not equal. Then there exists a point whose orbit contains
an open subset, since the $r(\so(n))$-orbits are already codimension 1.
This implies the sectional curvature is constant on such an open subset.
But this is impossible by Theorem \ref{sect-curvature-n}.
Notice that, for the calculation of the sectional curvature, we only need
the inequality above.
\end{proof}


For $a\in A$ and $k\in\so(n-1)\x\so(1) \subset K=\so(n)$,
$$ak = ka$$
and
$$
(K a)\cdot k = K ka = K a
$$
so that the stabilizer of $r(\so(n))$ at $a=\varphi(\bzero,y),\ y\not=1,\ y
\in \bbr^+$, contains $\so(n-1)\x\so(1)$. 
Let $S$ be the only subgroup of $K=\son$ properly containing
$\so(n-1)\x\so(1)$. Then $\so(n-1)\x\so(1)$ has index 2 in $S$, and no
element of $S-\so(n-1)\x\so(1)$ can fix $a$. Thus, we have

\begin{corollary}
\label{son-stablizer} For the $r(\so(n))$-action on
{$\varphi\inv(NA)=\bbr^{n-1}\x\bbr^+$}, the stabilizer at
$(\bzero,y),\ y\not=1$, is $\so(n-1)\x\so(1)$.
\end{corollary}

This can also be proved from the similar fact on $\bbh^n$ using the
weak $\son$-equivariant map.

\bigskip

\step[Embedding of $\so_0(2,1)$ into $\so_0(n,1)$]
Consider the subgroup $\soto$ of $\sono$, as
$$
I_{n-2}\x\soto\subset\sono,
$$
where $I_{n-2}$ is the identity matrix of size $n-2$.
For $k\in\ell(K)$ and $p\in\soto$, $k\cdot p\in\soto$ if and only if
$k\in\soto$.
Therefore, the space $\ell(\sot)\bs\soto$ is isometrically embedded into
$\ell(K)\bs\sono$.
With this embedding, there is an orthonormal basis for this 2-dimensional
subspace:
\begin{align*}
\bw_{n-1}&=c \ \frac{\partial}{\partial x_{n-1}} \Big|_{(\bzero,y)}\\
\bw_n&=y \frac{\partial}{\partial y} \Big|_{(\bzero,y)}\\
\end{align*}
where {$c=-\sqrt{\cosh(2\ln y)}$.}
\bigskip

\step[Orthonormal basis of $\SP^n$]
The right action of a matrix $k=\exp(\tfrac{\pi}{2}\cdot
E_{j,n-1})\in K$ ($j<n$) maps $(\bx, y)=(x_1,\dots,x_j,\dots,$
$x_{n-1},$ $y)\in\SP^n$ to $(\bx' , y)=(x_{1},\dots,x_{n-1},\dots,
-x_{j},y)\in\SP^n$. (i.e., exchanges the $(n-1)$st and $j$th slot).
More precisely, $\varphi(\bx,y)\cdot k=k'\cdot \varphi(\bx
',y)$ in $\so_0(n,1)$ for some $k'\in \son$. By applying such a
right action on $NA$ for $j=1,2,\dots,n-2$, we get the orthonormal
system at $(0,\dots,0,y)\in\SP^n$ with $y>1$:
\begin{align*}
\bw_1&=c \ \frac{\partial}{\partial x_1} \Big|_{(\bzero,y)}\\
\bw_2&=c \ \frac{\partial}{\partial x_{2}} \Big|_{(\bzero,y)}\\
\bw_3&=c \ \frac{\partial}{\partial x_{3}} \Big|_{(\bzero,y)}\\
&\cdots\\
\bw_{n-1}&=c \ \frac{\partial}{\partial x_{n-1}} \Big|_{(\bzero,y)}\\
\bw_n&=y \ \frac{\partial}{\partial y} \Big|_{(\bzero,y)}\\
\end{align*}
where {$c=-\sqrt{\cosh(2\ln y)}$.}
As before, we denote the upper half-space $\bbr^{n-1}\x\bbr^+$
with this metric by $\SP^n$. The above shows that the
metric is very close to being conformal to the standard $\bbr^n$.
\bigskip

\step
Recall that both spaces $\SP^n=\so(n)\bs\so_0(n,1)$ and
$\bbh^n=\so_0(n,1)/\so(n)$ have isometric actions by the maximal
compact subgroup, $r(K)$ and $\ell(K)$, respectively. The latter has
more isometries, $\ell(\so_0 (n,1))$.

\begin{proposition}
The map
$$
\tau:\ \SP^n\lra \bbh^n
$$
{\rm(}as a map $\bbr^{n-1}\x\bbr^+ \lra \bbr\x\bbr^+${\rm)}
defined by
$$
\tau(\bx,y)=(-\tfrac{\bx}{y},\tfrac{1}{y})
$$
has the following properties:

{\rm (1)}
$\tau$ is a weakly $\so(n)$-equivariant diffeomorphism of period 2.
More precisely,
$$
\tau(r(z)\cdot(\bx,y))=\ell(z\inv)\cdot\tau(\bx,y)
$$
for $z\in\so(n)$.
In other words, the identification of $\SP^n,
\bbh^n, NA, \text{ and } \bbr^{n-1} \times \bbr^+$ as sets permits
the following abuse of $\tau$ and gives a following relation between
$r(K)$-action and $\ell(K)$-action:
$r(\zh)\cdot(x,y)=\tau\left(\ell(\zh\inv)\cdot\tau(x,y)\right)$.

{\rm (2)} $\tau$ leaves the geometric spheres centered at
{$\bi=(\bzero,1)$} in each geometry invariant. That is, in both
geometries, for $m>0$, the Euclidean sphere
$$
|\mathbf x|^2+(y-\cosh(\ln m))^2=\sinh^2(\ln m)
$$
is a geometric sphere centered at $\bi=(\bzero,1)$ with radius $|\ln
m|$, in both geometries, and the map $\tau$ maps such a sphere to
itself. These spheres are $r(K)$-orbits in $\SP^n$ and
$\ell(K)$-orbits in $\bbh^n$ (when $\SP^n$ and $\bbh^n$ are
identified with $\bbr^{n-1}\x\bbr^+$) at the same time.

{\rm (3)} $\tau$ gives a 1-1 correspondence
between the two sets of all the geodesics passing through $\bi$
in the two geometries $\SP^n$ and $\HY^n$. 
\end{proposition}

\step
For the  $\ell(K)$-action on the hyperbolic space $\HY^n=G/K$, we can
take the ray $\{\bzero\}\x (0,1]$ as a cross section to the
$\ell(K)$-action. Clearly, $\{\bzero\}\x [1,\infty)$ is another cross
section. The cross section to the $r(K)$-action on $\SP^n=K\bs G$ is
the ray {$\{\bzero\}\x [1,\infty)$}. The action has a fixed point
$\bi=(\bzero,1)$, and all the other orbits are $\so(n)/\so(n-1) \cong
S^{n-1} \cong \so(n-1) \bs \so(n)$. The geometry of the whole space
$\SP^n=K\bs G$ is completely determined by the geometry on the line
$\{\bzero\}\x [1,\infty)$ as shown below.

\begin{theorem}
\label{warped-prod}
The space $\SP^n-\{\bi\}$ is isometric to the warped product
$(1,\infty) \times_{e^{2 \phi}} S^{n-1}$, where
$(1,\infty)$ has the induced metric from
$\{\bzero\}\x(1,\infty)\subset\SP^n$, that is,
$|\tfrac{\partial}{\partial t}(t_0)| = \tfrac{1}{t_0}$
for $t_0 \in (1, \infty)$;
$S^{n-1}$ has the standard metric;
and $e^{2 \phi (t)} = \frac{\sinh ^2 (\ln t)}{\cosh (2\ln t)}$.
\end{theorem}

\begin{proof}
The sphere $S^{n-1}\subset\bbr^n$ has a canonical $\so(n)$-action by matrix
multiplication.
Choose the north pole $\bn=(0, \dots, 0,1) \in S^{n-1}$ as a base point.
Then the $\so(n)$-action induces an action on $(1,\infty)\x S^{n-1}$,
acting trivially on the first factor.
The space $\SP^n-\{\bi\}$ also has an (isometric) action by $r(\so(n))$.
Using these actions, we define
$$
f: (1,\infty)\x S^{n-1} \lra \SP^n-\{\bi\}
$$
by
$$
f(t,a\cdot\bn)=f(a\cdot(t,\bn))=r(a\inv)\cdot(\bzero,t),
$$
where $(\bzero,t)\in \bbr^{n-1}\x\bbr^+\subset\SP^n$.
Since both actions have orbits $S^{n-1}$, and the stabilizers at
$(t,\bn)$ and $(\bzero,t)$ are both $\so(n-1)\x\so(1)\subset\so(n)$
(see Corollary \ref{son-stablizer}),
$f$ is well-defined, bijective and smooth.
\bigskip

Consider the subgroup $K_2$,
$$
K_2=I_{n-2}\x\so(2)\x I_1\subset\so(n)\x I_1\subset\so_0 (n,1).
$$
By taking the intersection of $(1,\infty)\x S^{n-1}$ and $\SP^n-\{\bi\}$
with the last 2-dimensional plane, we get isometric embeddings
$$
\CD
(1,\infty)\x S^{n-1} @>>> \SP^n-\{\bi\}\\
@A{\cup}AA @A{\cup}AA\\
(1,\infty)\x S^{1} @>>> \SP^2-\{\bi\}
\endCD
$$
Furthermore, when we give a warped product structure to $(1,\infty)\x S^{1}$
by the function $e^{2 \phi (t)} =\frac{\sinh ^2 (\ln t)}{\cosh (2\ln t)}$,
the restriction of the map $f$,
$$
\CD
(1,\infty)\x_{e^{2\phi}} S^{1} @>>> \SP^2-\{\bi\}
\endCD
$$
$$
f(t,a\cdot\bn)=f(a\cdot(t,\bn))=r(a\inv)\cdot(\bzero,t),
$$
where $(0,t)\in\SP^2\subset\SP^n$, becomes an isometry by Theorem
\ref{prop:warped-2-dim}.
Now it is clear that the $\so(n)$-action on both spaces make
the weakly equivariant map $f$ a global isometry.
Thus, the geometry of $\SP^n$ is completely determined by the geometry on
the cross section $\{\bzero\}\x(1,\infty)\subset\SP^n$
to the $r(K)$-action.
\end{proof}

\step
The sectional curvature of a plane containing the
$(1,\infty)$-direction in $(1,\infty) \times_{e^{2 \phi}} S^{n-1}$
is easy to calculate, since such a plane is a rotation of corresponding
plane for $\sot\bs\soto$ by $\son$.
Thus, the curvature of such a plane is exactly the same as the
2-dimensional case.

\bigskip

\step
For a general plane (not containing the $(1,\infty)$-direction), we
need some work. Notice that $\{f^{-1}_* \bw_1, \, \dots, \, f^{-1}_*
\bw_{n-1}, f^{-1}_* \bw_n\}$ is an orthonormal basis on $(1,\infty)
\times_{e^{2 \phi}} \{\bn\} $  such that $f^{-1}_* \bw_n$ is a normal
vector to each sphere and the others are tangent to the sphere. By
abusing notation, denote $f^{-1}_* \bw_i $ as $\bw_i$ again.

\begin{lemma}
\label{kappa-lemma}
For $(\mathbf 0,y) \in \SP^n-\{(\mathbf 0,1)\}
=(1,\infty) \times_{e^{2 \phi}} S^{n-1}$,  with $y >1,$ and
$
 \bw, \tilde{\bw} \in \mathrm{Span}\{ \bw_1, \dots \bw_{n-1} \} ,
$
with
$|\bw|_\phi = |\tilde{\bw}|_\phi = 1$
and $ \langle \bw, \tilde{\bw} \rangle_\phi = 0$, we have
$$
\kappa (a \bw_n + b \bw, \, c \bw_n + d \tilde{\bw}) =
(a^2 d^2 + b^2 c^2) \kappa(\bw_n , \bw) + b^2 d^2 \, \kappa (\bw, \tilde{\bw}).
$$
\end{lemma}

\begin{proof}
For tangent vectors $T_1,T_2,T_3\in T(S^{n-1})$ and $X\in T(S^{n-1})^\perp$
in the warped product, we
have
\begin{align*}
R(T_1, T_2)T_3 &=  R_{S^{n-1}} (T_1, T_2)T_3 -
  e^{2 \phi} \mid \! \nabla \phi \! \mid ^2
  \big( \langle T_2, T_3 \rangle _{S^{n-1}} T_1
  - \langle T_1, T_3 \rangle _{S^{n-1}} T_2 \big),\\
  R(X, T)Y &=  \big(
  h_{\phi} (X,Y) +
  \langle \nabla \phi , X \rangle \langle \nabla \phi , Y \rangle
  \big)
  T,
\end{align*}
see \cite[p.60, Proposition 2.2.2]{GW}. So,
$$
  \langle R (\tilde{\bw}, \bw)\bw, \bw_n \rangle_\phi = 0
\quad\text{and}\quad
  \langle R (\bw, \tilde{\bw}) \tilde{\bw}, \bw_n \rangle_\phi = 0,
$$
also
$$
\langle R(\bw_n,\bw)\bw_n, \tilde{\bw} \rangle _{\phi}
=  e^{2 \phi} \langle \bw, \tilde{\bw} \rangle _{S^1}
  \big(  h_{\phi} (\bw_n,\bw_n) +  \langle \nabla \phi , \bw_n \rangle ^2
  \big)
=0.
$$
Using an isometric $r(K)$-action rotating the $\{ \bw_n ,
\bw\}$-plane to $\{ \bw_n , \tilde{\bw} \}$-plane, we have $
\kappa(\bw_n, \bw) = \kappa(\bw_n, \tilde{\bw})$. Thus,
\begin{align*}
\kappa (a \bw_n + b \bw, \, c \bw_n + d \tilde{\bw})
  &=
   \langle
     R(a \bw_n + b \bw, \, c \bw_n + d \tilde{\bw}) (c \bw_n + d \tilde{\bw}),
     \, a \bw_n + d \bw
   \rangle_\phi  \\
  &= a^2 d^2 \kappa(\bw_n, \tilde{\bw}) + b^2 c^2 \kappa(\bw_n, \bw)
    + b^2 d^2 \kappa(\bw, \tilde{\bw}) \\
  & = (a^2 d^2 + b^2 c^2) \kappa(\bw_n , \bw) + b^2 d^2 \, \kappa (\bw,
\tilde{\bw}).
\qedhere
\end{align*}
\end{proof}

\begin{theorem}[
The sectional curvature of the space $\SP^n=\son\bs\sono$]
\label{sect-curvature-n}
For $(\mathbf 0,y) \in \SP^n-\{(\mathbf 0,1)\}
=(1,\infty) \times_{e^{2 \phi}} S^{n-1}$,  with $y >1$, let $\sigma$ be a
2-dimensional tangent plane at $(\mathbf 0,y)$ whose angle with the
$y$-axis is $\theta$. Then its sectional curvature
$\kappa(y,\theta):=\kappa(\sigma)$ is
$$
\kappa(y,\theta)=
  \cos ^2 \theta \, \frac{4 y^2 (1 + 3 y^2 + y^4)}{(1 + y^4)^2}
  + \sin ^2 \theta \,
    \frac{2 (1+ 2y^2 + 4y^4 + 2y^6 + y^8)}{(1 + y^4)^4}.
$$
This curvature formula is valid for all $1\leq y <\infty$.
Therefore $0<\kappa_{(\bzero,y)}\leq 5$ for all $y\geq 1$, and at $y=1$,
$\kappa(1,\theta)=5$ gives the maximum curvature for all $y\geq 1$.
\end{theorem}

\begin{proof}
It is obvious in the case of either $\theta = 0$ or
$\theta=\tfrac{\pi}{2}.$

Assume $0< \theta < \tfrac{\pi}{2}.$ Let $\hat{\bw}$ be the
orthogonal projection of $\bw_n$ to $\sigma$. There is a unique $\bw
\in T(S^{n-1})$, which lies in the plane $\{ \hat{\bw}, \bw_n \}$,
such that we can write $\hat{\bw}$ as a linear combination of $\bw_n$
and $\bw$ with respect to $\theta$: $\hat{\bw} = r \cos \theta \
\bw_n + r \sin \theta \ \bw$ for some $r> 0.$ Now let $\tilde{\bw}$
be a unit vector in $\sigma \cap T(S^{n-1})$.
Since $\tilde{\bw}, \hat{\bw} \in \sigma$,
$$
  0 = \langle \bw_n, \tilde{\bw} \rangle_\phi
    = \langle \hat{\bw}, \tilde{\bw} \rangle_\phi
    = \langle r \cos \theta \bw_n + r \sin \theta \bw, \, \tilde{\bw} \rangle_\phi
    = r \sin \theta \langle \bw, \tilde{\bw} \rangle_\phi,
$$
which implies
$$\langle \bw, \tilde{\bw} \rangle_\phi = 0$$
and from the above lemma
\begin{align*}
\kappa(y,\theta)
  &= \kappa(\hat{\bw}, \tilde{\bw})  \\
  &= \kappa(\cos \theta \ \bw_n + \sin \theta \ \bw, \, \tilde{\bw}) \\
  &= \cos ^2 \theta \, \kappa(\bw_n, \tilde{\bw})
     + \sin ^2 \theta \, \kappa(\bw, \tilde{\bw}) \\
  &= \cos ^2 \theta \, \kappa (y) + \sin ^2 \theta \, \kappa(\bw, \tilde{\bw}),
\end{align*}
where $\kappa(y)$ is the curvature of any tangent 2-plane containing
$\bw _n$. Now, we get
$$\mid \! \bw \! \mid_{S^{n-1}} = \mid \! \tilde{\bw} \! \mid_{S^{n-1}} = e^{- \phi (y)}$$
with respect to the standard metric on $S^{n-1}$ and, from the
formula of $R(T_1, T_2)T_3$ in the proof of Lemma \ref{kappa-lemma},
\\
\begin{align*}
  \kappa(\bw, \tilde{\bw})
  &= \langle R (\bw, \tilde{\bw})\tilde{\bw}, \bw \rangle_\phi \\
  &= e^{2 \phi (y)}
     \langle R (\bw, \tilde{\bw})\tilde{\bw}, \bw \rangle _{S^{n-1}} \\
  &= e^{2 \phi (y)}
     \big(
        \kappa_{S^{n-1}} (\bw, \tilde{\bw}) -
        e^{2 \phi (y)} \mid \! \nabla \phi \! \mid ^2
        (
          \mid \! \bw \! \mid ^2 _{S^{n-1}}
          \mid \! \tilde{\bw} \! \mid ^2 _{S^{n-1}} -
          \langle \bw , \tilde{\bw} \rangle _{S^{n-1}} ^2
        )
     \big) \\
  &= e^{2 \phi (y)}
     \big(
        e^{-4 \phi (y)} -
        e^{2 \phi(y)}
        \mid \! \nabla \phi  \! \mid  ^2 e^{-4 \phi (y)}
     \big) \\
  &= e^{- 2 \phi (y)} - \langle \nabla \phi , \bw _n \rangle ^2 \\
  &= \tfrac{\cosh (2 \ln y)}{\sinh ^2 (\ln y)} - \big( \bw_n (\phi)\big)^2 \\
  &= \tfrac{2 (y^4 +1)}{(y^2 -1)^2} -
     \big( y \tfrac{\partial \phi}{\partial y} \big)^2 \\
  &= \frac{2 (1+ 2y^2 + 4y^4 + 2y^6 + y^8)}{(1 + y^4)^2}.
\end{align*}
Thus,
$$
\kappa(y,\theta)=
  \cos ^2 \theta \, \frac{4 y^2 (1 + 3 y^2 + y^4)}{(1 + y^4)^2}
  + \sin ^2 \theta \, \frac{2 (1+ 2y^2 + 4y^4 + 2y^6 + y^8)}{(1 + y^4)^2}.
$$
By the remark after Proposition \ref{ON-on-r2}, by the continuity argument,
this curvature formula is valid even at the removed point $(\mathbf 0,1)$
with $\kappa_{(\bzero,1)}=5$.
\bigskip

To estimate the values $\kappa(y,\theta)$, let
\begin{align*}
f(y) &=  \frac{4 y^2 (1 + 3 y^2 + y^4)}{(1 + y^4)^2} \\
g(y) &=  \frac{2 (1+ 2y^2 + 4y^4 + 2y^6 + y^8)}{(1 + y^4)^2}
\end{align*}
for $y > 1$. Then
$$
0 < f(y) < 5 \quad \text{and} \quad 0 < g(y) < 5.
$$
The relation,
$$
\kappa(y,\theta)=
    \cos ^2 \theta \, f(y) + \sin ^2 \theta \, g(y)
    = \frac{f(y) + g(y) + \cos (2 \theta) \big( f(y) - g(y) \big)}{2}
$$
gives us the following inequality
$$
\frac{f(y) + g(y) - \mid \! f(y) - g(y) \! \mid}{2}\leq
\kappa(y,\theta)
\leq \frac{f(y) + g(y) + \mid \! f(y) - g(y) \! \mid}{2},
$$
so that
$$
\mathrm{min}\{f(y), g(y)\} \leq
\kappa(y,\theta)
\leq \mathrm{max}\{f(y), g(y)\},
$$
which shows $0<\kappa(y,\theta)<5$.
\end{proof}

\bigskip
\bigskip

\bibliographystyle{amsalpha}

\bigskip
\bigskip

\end{document}